\documentclass[a4paper]{amsart}
\usepackage{graphicx}
\usepackage{amssymb}
\usepackage{amsmath}
\usepackage{amsthm}
\usepackage{amscd}
\usepackage[all,2cell]{xy}

\usepackage[pagebackref,colorlinks]{hyperref}

\UseAllTwocells \SilentMatrices
\newtheorem{thm}{Theorem}[section]

\newtheorem{cor}[thm]{Corollary}

\newtheorem{lem}[thm]{Lemma}
\newtheorem{exm}[thm]{Example}

\newtheorem{prop}[thm]{Proposition}

\theoremstyle{definition}
\newtheorem{defn}[thm]{Definition}
\theoremstyle{remark}
\newtheorem{rem}[thm]{\bf Remark}
\numberwithin{equation}{section}

\begin{document}
\title[The singularity category via the stabilization]{The singularity category via the stabilization}
\author[Xiao-Wu Chen] {Xiao-Wu Chen}

\makeatletter
\@namedef{subjclassname@2020}{\textup{2020} Mathematics Subject Classification}
\makeatother

\date{\today}
\subjclass[2020]{18G65, 18G80, 13D02, 16S88}

\thanks{E-mail: xwchen$\symbol{64}$mail.ustc.edu.cn}
\keywords{singularity category, stable category, stabilization, singular equivalence, Leavitt ring}

\begin{abstract}
We give a detailed proof of the following fundamental result: the singularity category of a ring is triangle equivalent to the stabilization of its stable module category. The result yields singular equivalences between rings of  different nature. We use Leavitt rings to describe singularity categories of artinian rings. 
\end{abstract}

\maketitle

\dedicatory{}%
\commby{}%

\section{Introduction}

Let $R$ be a unital ring. The singularity category \cite{Buc, Orl} of $R$ detects its homological singularity in the following sense: the singularity category vanishes if $R$ has finite global dimension.  The study of singularity categories  is closely related to Cohen-Macaulay modules \cite{Buc, Hap91}, noncommutative algebraic geometry \cite{Orl2, Len} and homological mirror symmetry of Landau-Ginzburg models \cite{Orl, Ebel}. 

The singularity category  is originally called the \emph{stabilized derived category} in \cite{Buc}, as it describes  the stable homological features of $R$. It captures the asymptotic  behavior of the syzygy endofunctor $\Omega_R$ on the stable module category. This statement is made precise by the  fundamental result in \cite{Buc, KV}, which  states that the singularity category  is triangle equivalent to the stabilization \cite{Freyd, Hel68} of the stable module category by formally inverting $\Omega_R$; see also \cite{Bel}. 

The goal of the paper is twofold. The first is to give a detailed proof of the fundamental result above. The proof presented in \cite{Bel} seems leaving  out some technical details. The second is to report some applications. The fundamental result yields another proof of Buchweitz's theorem \cite{Buc}, which describes the singularity category of a Gorenstein ring using Gorenstein-projective  modules \cite{EJ}; compare \cite{CW}. We report recent work \cite{Kalck, Wei} on singular equivalences between rings of rather different nature, and a new description \cite{CW25} of the singularity category of an artinian ring using Leavitt rings \cite{CO, CWKW}.

The paper is structured as follows. In Section~2, we recall the stabilization of a looped category. In Section~3, we recall from \cite{KV, Bel} that the stabilization of a left triangulated category is naturally a triangulated category.  In Section~4, we give a detailed proof of the fundamental result mentioned above; see Theorem~\ref{thm:BKVB}. We give its applications in the last two sections. 

We mention that almost all the results are already contained in the literature. The only exception might be Theorem~\ref{thm:radical}, which describes the singularity category of  an artinian ring with radical square zero via a Leavitt ring. It is a slight generalization of a result in \cite{Smi} with a  different proof.  For quivers and stable module categories, we refer to \cite{ARS}. For triangulated categories, we refer to \cite{Ver, Hap}.

\section{The stabilization of a looped category}

In this section, we recall basic facts on the stabilization. Roughly speaking, the stabilization is a categorical construction of formally inverting an endofunctor, which is due to \cite[Section~2]{Freyd} and \cite[Chapter~I]{Hel68}. A similar idea is traced back to the well-known Spanier-Whitehead category \cite{SW} of pointed CW-complexes. We mention its connection to spectra \cite{Gra, Bel, Boo}.

By a \emph{looped category} \cite{BM, Bel}, we mean a pair $(\mathcal{C}, \Omega)$ consisting of a category $\mathcal{C}$ and an endofunctor $\Omega$ on $\mathcal{C}$. The looped category is called \emph{stable} if $\Omega$ is an autoequivalence on $\mathcal{C}$,  and called \emph{strictly stable} if $\Omega$ is an automorphism on $\mathcal{C}$.

Let $(\mathcal{C}, \Omega)$ and $(\mathcal{D}, \Delta)$ be two looped categories. A \emph{looped functor}
$$(F, \delta)\colon (\mathcal{C}, \Omega)\longrightarrow (\mathcal{D}, \Delta)$$
consists of a functor $F\colon \mathcal{C}\rightarrow \mathcal{D}$ and a natural isomorphism $\delta\colon F\Omega \rightarrow \Delta F$.  For such a looped functor $(F, \delta)$ and $n\geq 1$, we define inductively a natural isomorphism
\begin{align}\label{equ:delta}
\delta^{(n)} \colon  F\Omega^n \stackrel{\sim} \longrightarrow \Delta^n F
\end{align}
by $\delta^{(1)}=\delta$ and $\delta^{(n+1)}=\Delta^n \delta\circ \delta^{(n)}\Omega$. We set $\delta^{(0)}$ to be the identity transformation on $F$, where $\Omega^0$ and $\Delta^0$ are understood to be the corresponding identity endofunctors. A  looped functor $(F, \delta)$ is called \emph{strictly looped} if $F\Omega=\Delta F$ and $\delta$ is the identity transformation on $F\Omega$. In this situation, $(F, \delta)$ is abbreviated as $F$.

In what follows, we assume that  $(\mathcal{C}, \Omega)$ is a looped category. We define its \emph{stabilization} $\mathcal{S}=\mathcal{S}(\mathcal{C}, \Omega)$ as follows. The objects are given by pairs $(X, n)$, which consist of an object $X$ in $\mathcal{C}$ and an integer $n$. The Hom-set from $(X, n)$ to $(Y, m)$ is given by a colimit
$$\mathcal{S}((X, n), (Y, m))={\rm colim}\; \mathcal{C}(\Omega^{p-n}(X), \Omega^{p-m}(Y)),$$
where $p$ runs over all the integers satisfying $p\geq {\rm max}\{n, m\}$, and the structure map is induced by the endofunctor $\Omega$. 

For a morphism $f\in \mathcal{C}(\Omega^{p-n}(X), \Omega^{p-m}(Y))$, its image in $\mathcal{S}((X, n), (Y, m))$ will be denoted by 
$$\iota_p(f)\colon (X, n)\longrightarrow (Y, m).$$
We emphasize that the morphism $\iota_p(f)$ really  depends on $n$ and $m$.  By the very definition, we have
\begin{align}\label{equ:omega}
    \iota_p(f)=\iota_{p+k}(\Omega^k(f))
\end{align}
for any $k\geq 0$.

The composition in $\mathcal{S}$ is induced by the one in $\mathcal{C}$. To be more precise, we take any two morphisms $\iota_p(f)\colon (X, n)\rightarrow (Y, m)$ and  $\iota_q(g)\colon (Y, m)\rightarrow (Z, l)$. By (\ref{equ:omega}), we may always assume that $q\geq p$. The composition is defined  by
$$\iota_q(g)\circ \iota_p(f)= \iota_{q} (g\circ \Omega^{q-p}(f)).$$

The category $\mathcal{S}$ carries an automorphism $\Sigma$ defined by $\Sigma(X, n)=(X, n+1)$ on objects. It sends $\iota_p(f)\colon (X, n)\rightarrow (Y, m)$ to $\iota_{p+1}(f)\colon (X, n+1)\rightarrow (Y, m+1)$ on morphisms. This automorphsim $\Sigma$ is called the \emph{suspension functor} of $\mathcal{S}$.  Consequently, we have a strictly stable category $(\mathcal{S}, \Sigma^{-1})$.

We have a canonical functor $\mathbf{S}\colon \mathcal{C}\rightarrow \mathcal{S}$,  which sends $A$ to $(A, 0)$, and sends a morphism $f\colon A\rightarrow B$ to $\mathbf{S}(f)=\iota_0(f)\colon (A, 0)\rightarrow (B, 0)$.

\begin{lem}\label{lem:stab}
    Consider the morphism $\iota_p(f) \colon (X, n)\rightarrow (Y, m)$ above. The following statements hold.
    \begin{enumerate}
    \item  For each $s\geq 0$, the morphism $\iota_{n+s}({\rm Id}_{\Omega^s(X)})\colon (X, n)\rightarrow (\Omega^s(X), n+s)$ is an isomorphism.
    \item We have a commutative diagram in $\mathcal{S}$.
    \[\xymatrix{
    (X, n) \ar[d]_-{\iota_{p}({\rm Id}_{\Omega^{p-n}(X)})} \ar[rr]^-{\iota_p(f)} && (Y, m) \ar[d]^-{\iota_p({\rm Id}_{\Omega^{p-m}(Y)})}\\
    (\Omega^{p-n}(X), p) \ar[rr]^-{\Sigma^{p} \mathbf{S}(f)} && (\Omega^{p-m}(Y), p)
    }\]
    \end{enumerate}
\end{lem}

\begin{proof}
    For (1), we mention that  the inverse of the given morphism $\iota_{n+s}({\rm Id}_{\Omega^s(X)})$ is given by $\iota_{n+s}({\rm Id}_{\Omega^s(X)})\colon (\Omega^s(X), n+s)\rightarrow (X, n)$. The commutativity of the diagram in (2) is trivial, by using $\Sigma^{p} \mathbf{S}(f)=\iota_p(f)$.
\end{proof}

The following remark on \emph{enlarging the second entries} makes the stabilization $\mathcal{S}$ more accessible.

\begin{rem}\label{rem:enlarge}
By the isomorphism in Lemma~\ref{lem:stab}(1), we may always enlarge the second entry of any object in $\mathcal{S}$. The commutative square above implies that up to isomorphism, any morphism in $\mathcal{S}$ is of the form $\Sigma^{p}\mathbf{S}(f)$ for some morphism $f$ in $\mathcal{C}$. For the same reason, any finite commutative diagram in $\mathcal{S}$ is isomorphic to the one, which is obtained by applying $\Sigma^p\mathbf{S}$ to a corresponding commutative diagram in $\mathcal{C}$ for a sufficiently large $p$. 
\end{rem}

For each object $X$ in $\mathcal{C}$, we have a natural isomorphism
$$\theta_X=\iota_0({\rm Id}_{\Omega(X)})\colon \mathbf{S}\Omega(X)=(\Omega(X), 0) \stackrel{\sim}\longrightarrow (X, -1)=\Sigma^{-1}\mathbf{S}(X).$$
In other words, we have a looped functor
$$(\mathbf{S}, \theta)\colon (\mathcal{C}, \Omega)\longrightarrow (\mathcal{S}, \Sigma^{-1}),$$
called the \emph{stabilization functor}. In a certain sense, it formally inverts the endofunctor $\Omega$ on $\mathcal{C}$.

The stabilization functor $(\mathbf{S}, \theta)$ enjoys the following universal property; see \cite[Proposition~2.1]{Freyd}  and \cite[Proposition~1.1]{Hel68}.

\begin{prop}\label{prop:univ}
Let $(F, \delta)\colon (\mathcal{C}, \Omega)\rightarrow (\mathcal{D}, \Delta)$ be a looped functor with $(\mathcal{D}, \Delta)$ a strictly stable category.  Then there is a unique functor $\tilde{F}\colon (\mathcal{S}, \Sigma^{-1})\rightarrow (\mathcal{D}, \Delta)$, which is strictly looped satisfying $F=\tilde{F}\mathbf{S}$ and $\delta=\tilde{F}\theta$.
\end{prop}

\begin{proof}
The functor $\tilde{F}$ sends $(X, n)$ to $\Delta^{-n}F(X)$. For any morphism $\iota_p(f)\colon  (X, n)\rightarrow (Y, m)$ in $\mathcal{S}$, we have
\begin{align*}
\tilde{F}(\iota_p(f))=\Delta^{-p}(\delta^{(p-m)}_Y \circ F(f)\circ (\delta^{(p-n)}_X)^{-1})\colon \Delta^{-n} F(X)\longrightarrow \Delta^{-m} F(Y).
\end{align*}
Here, we are using the natural isomorphisms  (\ref{equ:delta}). 
\end{proof}

The following useful results are due to  \cite[1.2~Proposition]{Tier}; compare  \cite[Proposition~3.4]{Bel}. 

\begin{prop}\label{prop:equiv}
Consider  the strictly looped functor $\tilde{F}\colon (\mathcal{S}, \Sigma^{-1})\rightarrow (\mathcal{D}, \Delta)$ above. Then the following statements hold.
\begin{enumerate}
\item  The functor $\tilde{F}$ is full if and only if for any objects $X$ and $Y$ in $\mathcal{C}$ and any  morphism $g\colon F(X)\rightarrow F(Y)$ in $\mathcal{D}$, there exist $i\geq 0$ and a morphism $f\colon \Omega^i(X)\rightarrow \Omega^i(Y)$ in $\mathcal{C}$ satisfying $\Delta^i(g)=\delta^{(i)}_Y\circ F(f)\circ (\delta^{(i)}_X)^{-1}$.
\item  The functor $\tilde{F}$ is faithful  if and only if for any two morphisms $f, f'\colon X\rightarrow Y$ in $\mathcal{C}$ with $F(f)=F(f')$, there exists $i\geq 0$  such that $\Omega^i(f)=\Omega^i(f')$.
\item  The functor $\tilde{F}$ is dense if and only if for any object $Z$ in $\mathcal{D}$, there exist $i\geq 0$ and an object $X$ in $\mathcal{C}$ satisfying $\Delta^i(Z)\simeq F(X)$.
\end{enumerate}
\end{prop}

\begin{proof}
We only prove (1). The ``only if" part is trivial, since the fullness implies that $g=\tilde{F}(\iota_i(f))$  for some $i\geq 0$ and some morphism $f$.  For  the ``if" part, we take two objects $(X, n)$ and $(Y, m)$ in $\mathcal{S}$, and any morphism $g'\colon \tilde{F}(X, n)\rightarrow \tilde{F}(Y, m)$. By Remark~\ref{rem:enlarge}, we may assume that $n=m$. There is a unique morphism $g\colon F(X) \rightarrow F(Y)$ satisfying $g'=\Delta^{-n}(g)$. By assumption, there exists a morphism $f\colon \Omega^i(X)\rightarrow \Omega^i(Y)$ in $\mathcal{C}$ satisfying $\Delta^i(g)=\delta^{(i)}_Y\circ F(f)\circ (\delta^{(i)}_X)^{-1}$. The morphism $\iota_{n+i}(f)\colon (X, n)\rightarrow (Y, n)$ in $\mathcal{S}$ is well defined. It is routine to verify that $\tilde{F}(\iota_{n+i}(f))=g'$.
\end{proof}

\section{The stabilization of a left triangulated category}

In this section, we recall the notion of a left triangulated category \cite{KV, BM} and its stabilization. Left triangulated categories are one-sided analogues to the well-known triangulated categories in the sense of \cite{Ver, DP}.  

The following definition is taken from \cite[Section~2]{BM}; compare \cite[Subsection~1.1]{KV}.

\begin{defn}
    A \emph{left triangulated category} is a triple $(\mathcal{C}, \Omega, \mathcal{E})$, where $\mathcal{C}$ is an additive category, $\Omega$ is an additive endofunctor on $\mathcal{C}$, and $\mathcal{E}$ is a chosen class of sequences $\Omega(Z) \stackrel{h}\rightarrow X \stackrel{g}\rightarrow Y \stackrel{f}\rightarrow Z$ in $\mathcal{C}$, which are called \emph{exact left-triangles}. These data are subject to the following axioms. 
    \begin{enumerate}
    \item[(LTR1)] The class $\mathcal{E}$ of  exact left-triangles are closed under isomorphisms. For each object $X$, the sequence $\Omega(0) \rightarrow X \stackrel{{\rm Id}_X} \rightarrow X \rightarrow 0$ belongs to $\mathcal{E}$. Each morphism $f\colon Y\rightarrow Z$ fits into an exact left-triangle  $\Omega(Z) \rightarrow X \rightarrow Y \stackrel{f}\rightarrow Z$. 
    \item[(LTR2)]  The \emph{rotation axiom}: for each member $\Omega(Z) \stackrel{h}\rightarrow X \stackrel{g}\rightarrow Y \stackrel{f}\rightarrow Z$ in $\mathcal{E}$, the rotated one $\Omega(Y) \stackrel{-\Omega(f)}\rightarrow \Omega(Z) \stackrel{h}\rightarrow X \stackrel{g}\rightarrow Y$ also belongs to $\mathcal{E}$.
\item[(LTR3)] The \emph{morphism axiom}: for any diagram with rows belonging to $\mathcal{E}$ and $a\circ f=f'\circ b$, 
\[\xymatrix{
\Omega(Z)\ar[d]_-{\Omega(a)} \ar[r]^-{h} & X \ar[r]^-g & Y\ar[d]^-{b} \ar[r]^-f & Z\ar[d]^-{a}\\
\Omega(Z') \ar[r]^-{h'} & X' \ar[r]^-{g'} & Y' \ar[r]^-{f'} & Z'
}\]
there exist a morphism $c\colon X\rightarrow X'$ making the diagram commute. 
\item[(LTR4)] The \emph{octahedral axiom}: for any composable pair of  morphisms $X\stackrel{v} \rightarrow Y \stackrel{u} \rightarrow Z$, there exists a commutative diagram whose  two lower rows and two middle columns belong to $\mathcal{E}$.
\[\xymatrix{
 & \Omega(C)\ar[d] \ar[r]^-{\Omega(c)} & \Omega(Y) \ar[d]^-{\omega}\\
 \Omega(Y)\ar[d]_-{\Omega(c)} \ar[r]^-{\omega} & A\ar@{.>}[d] \ar@{=}[r] & A\ar[d] \\
 \Omega(Z)\ar@{=}[d] \ar[r] &  B\ar@{.>}[d]  \ar[r] & X \ar[d]^-{v}\ar[r]^-{u\circ v} & Z\ar@{=}[d]\\
 \Omega(Z) \ar[r] &  C  \ar[r]^-{c} & Y \ar[r]^-{u} & Z\\
}\]
    \end{enumerate}
\end{defn}

\begin{rem}\label{rem:left}
(1) For the diagram in (LTR4), the existence of all solid arrows  follows from (LTR1) and (LTR3). The essential content of the octahedral axiom is  to require the existence of the two dotted arrows, which make the commutativity such that the resulting  sequence $\Omega(C) \rightarrow A \rightarrow B \rightarrow C$ belongs to $\mathcal{E}$.   

(2) Dually, one defines \emph{right triangulated categories}, which are called suspended categories in \cite{KV}. We menion the more subtle notion of a pre-triangulated category in \cite[Definition~4.9]{Bel01}; compare \cite[Definition~6.5.2]{Hov}.

(3) When the endofunctor $\Omega$ is an autoequivalence, that is, $(\mathcal{C}, \Omega)$ is stable, the left triangulated category $(\mathcal{C}, \Omega, \mathcal{E})$ gives rise to a  triangulated category in a natural manner. In practice, we identify a stable left triangulated category with a triangulated category. 
\end{rem}

In what follows,  $(\mathcal{C}, \Omega, \mathcal{E})$ is a fixed left triangulated category.  Let $(\mathcal{D}, \Delta, \mathcal{E}')$ be another left triangulated category. A \emph{triangle functor} 
$$(F, \delta)\colon (\mathcal{C}, \Omega, \mathcal{E})\longrightarrow (\mathcal{D}, \Delta, \mathcal{E}')$$
is a looped functor such that for each member  $\Omega(Z) \stackrel{h}\rightarrow X \stackrel{g}\rightarrow Y \stackrel{f}\rightarrow Z$ in $\mathcal{E}$, the resulting sequence
$$ \Delta F(Z) \xrightarrow{F(h)\circ \delta^{-1}_Z}  F(X) \stackrel{F(g)}\longrightarrow F(Y) \stackrel{F(f)}\longrightarrow F(Z)$$
belongs to $\mathcal{E}'$. 

Consider the stabilization $\mathcal{S}=\mathcal{S}(C, \Omega)$ and the suspension functor  $\Sigma$ on it. For each member  $\Omega(Z) \stackrel{h}\rightarrow X \stackrel{g}\rightarrow Y \stackrel{f}\rightarrow Z$ in $\mathcal{E}$ and any integer $p$, we define a \emph{standard triangle} in $\mathcal{S}$ as follows.
$$ (X, p) \stackrel{\iota_p(g)}\longrightarrow (Y, p) \stackrel{\iota_p(f)}\longrightarrow (Z, p) \stackrel{-\iota_{p+1}(h)}\longrightarrow (X, p+1)=\Sigma(X, p)$$
Denote by $\tilde{\mathcal{E}}$ the isomorphism closure of the class of standard triangles. The members of $\tilde{\mathcal{E}}$ are called \emph{exact triangles}. 

The following result is due to \cite[2.1~Proposition]{KV} and \cite[Subsection~3.1]{Bel}; compare \cite[Subsection~4.4]{Gra}.

\begin{prop}\label{prop:stab-tri}
    Let $(\mathcal{C}, \Omega, \mathcal{E})$ be a left triangulated category. Then $(\mathcal{S}, \Sigma, \tilde{\mathcal{E}})$ is a triangulated category. Moreover, the stabilization functor
    $$(\mathbf{S}, \theta)\colon (\mathcal{C}, \Omega, \mathcal{E})\longrightarrow (\mathcal{S}, \Sigma^{-1}, \tilde{\mathcal{E}})$$
is a triangle functor.
\end{prop}

The resulting triangulated category $(\mathcal{S}(\mathcal{C}, \Omega), \Sigma, \tilde{\mathcal{E}})$ might be denoted by $\mathcal{S}(\mathcal{C}, \Omega, \mathcal{E})$.

\begin{proof}
    We have to verify the axioms (TR1)-(TR4) for the triple $(\mathcal{S}, \Sigma, \tilde{\mathcal{E}})$. Recall from Remark~\ref{rem:enlarge} that any morphism in $\mathcal{S}$ is isomorphic to the one of the form $\iota_p(f)\colon (A, p)\rightarrow (B, p)$ for some morphism $f\colon A\rightarrow B$ in $\mathcal{C}$. Then (TR1) follows from (LTR1) of $\mathcal{C}$.

    For (TR2), we may start with a standard triangle. Then (TR2), the rotation axiom in both directions,  follows immediately from (LTR2) of $\mathcal{C}$. For (TR3), we may assume that the following two given triangles are both standard, and that the middle square is commutative. 
\[
\xymatrix{(X, p) \ar[r]^-{\iota_p(g)} & (Y, p)\ar[d]_-{\eta} \ar[r]^-{\iota_p(f)} & (Z, p)\ar[d]^-{\kappa} \ar[r]^-{-\iota_{p+1}(h)} & (X, p+1)\\ 
(X', q) \ar[r]^-{\iota_q(g')} & (Y', q) \ar[r]^-{\iota_q(f')} & (Z', q) \ar[r]^-{-\iota_{q+1}(h')} & (X', q+1)
}\]
By Remark~\ref{rem:enlarge} again, we may assume that $q=p$. Moreover, we may assume that $\kappa=\iota_p(a)$ and $\eta=\iota_p(b)$. By possibly enlarging the second entries again,  the middle commutative  square implies  that $a\circ f=f'\circ b$ in $\mathcal{C}$. Applying (LTR3), we obtain a morphism $c\colon X\rightarrow X'$ in $\mathcal{C}$. Then $\iota_p(c)$ is a required morphism.

By Remark~\ref{rem:enlarge}, any composable pair in $\mathcal{S}$ is isomorphic to the one of the form 
    $$ (A, p)\stackrel{\iota_p(f)}\longrightarrow (B, p) \stackrel{\iota_p(g)} \longrightarrow (C, p)$$
    for some composable pair $(f, g)$ in $\mathcal{C}$ and a sufficiently large integer $p$. Then (TR4) follows from (LTR4) of $\mathcal{C}$ and necessary rotations.
\end{proof}

\begin{rem}
    We mention that certain differential graded enhancements of left triangulated categories and their stabilizations are investigated in \cite{XF}.
\end{rem}

The stabilization functor $(\mathbf{S}, \theta)$ above enjoys a universal property. Let $(\mathcal{D}, T, \mathcal{E}')$ be a triangulated category with the suspension functor $T$ being an automorphism. Consider a triangle functor $(F, \delta)\colon (\mathcal{C}, \Omega, \mathcal{E}) \rightarrow (\mathcal{D}, T^{-1}, \mathcal{E}')$.  Here, $ (\mathcal{D}, T^{-1}, \mathcal{E}')$ is the corresponding left triangulated category of $(\mathcal{D}, T, \mathcal{E}')$; see Remark~\ref{rem:left}(3).  

\begin{cor}\label{cor:trifun}
    Keep the assumptions above. Then there is a unique functor 
    $$\tilde{F}\colon (\mathcal{S}, \Sigma, \tilde{\mathcal{E}}) \longrightarrow (\mathcal{D}, T, \mathcal{E}'),$$
    which is a strictly triangle functor between triangulated categories satisfying $F=\tilde{F}\mathbf{S}$ and $\delta=\tilde{F}\theta$.
\end{cor}

\begin{proof}
    It suffices to observe that $\tilde{F}$ sends standard triangles in $\mathcal{S}$ to exact triangles in $\mathcal{T}$, using the assumption that $(F, \delta)$ is a triangle functor. 
\end{proof}

In what follows, we write $\mathcal{C}=(\mathcal{C}, \Omega, \mathcal{E})$. The stabilization $\mathcal{S}(\mathcal{C}, \Omega, \mathcal{E})$ will be denoted by $\mathcal{S}(\mathcal{C})$.

Let $(H, \sigma)\colon \mathcal{C}\rightarrow \mathcal{C}'$ be  a triangle functor between two left triangulated categories. By the universal property above, there is a strictly triangle functor $\mathcal{S}(H, \sigma)$ making the following square commute.
\[\xymatrix{
\mathcal{C}\ar[d]_-{(\mathbf{S}, \theta)} \ar[rr]^-{(H, \sigma)} && \mathcal{C}' \ar[d]^-{(\mathbf{S}, \theta)}\\
\mathcal{S}(\mathcal{C}) \ar@{.>}[rr]^-{\mathcal{S}(H, \sigma)} && \mathcal{S}(\mathcal{C}') 
}\]

The following notion is taken from \cite[Section~2]{Chen18}.

\begin{defn}\label{defn:pte}
     The triangle functor $(H, \sigma) \colon  \mathcal{C}\rightarrow \mathcal{C}' $ is called a \emph{pre-triangle equivalence}, provided that the strictly triangle functor $\mathcal{S}(H, \sigma)$ between the stabilizations is an equivalence, and thus a triangle equivalence. 
\end{defn}

The following example is about semisimple (left) triangulated categories; compare \cite[Lemma~2.1]{Kalck} and \cite[Lemma~3.4]{Chen11}. 

\begin{exm}\label{exm:ss}
Let $\mathcal{C}$ be a \emph{semisimple} additive category, that is, any morphism $X\rightarrow Y$ is isomorphic to the direct sum of ${\rm Id}_A\colon A\rightarrow A$ and $0\colon B\rightarrow C$   for some objects $A, B, C$. Let $\Omega$ be any additive endofunctor on $\mathcal{C}$. Denote by $\mathcal{E}^{\rm tr}$ the isomorphism closure of direct sums the following \emph{trivial} left-triangles: $\Omega(A) \rightarrow 0\rightarrow A \stackrel{{\rm Id}_A}\rightarrow A$ and $\Omega(C) \xrightarrow{\binom{0}{{\rm Id}_{\Omega(C)}}} B\oplus \Omega(C) \xrightarrow{({\rm Id}_B, 0)}  B \stackrel{0}\rightarrow C$.  Then $(\mathcal{C}, \Omega, \mathcal{E}^{\rm tr})$ is a left triangulated category. Since $\mathcal{E}^{\rm tr}$ is the unique left triangulated structure on $(\mathcal{C}, \Omega)$, we will denote the resulting left triangulated category simply by  $(\mathcal{C}, \Omega)$.

By Lemma~\ref{lem:stab}(2), the stablization $(\mathcal{S}, \Sigma, \tilde{\mathcal{E}}^{\rm tr})$ is also semisimple. In other words, any exact triangle in $\mathcal{S}$ is isomorphic to direct sums of trivial ones. 
\end{exm}

\section{The stable module category and singularity category}

In this section, we give a detailed proof on the following fundamental result: the big singularity category of an arbitrary ring is triangle equivalent to the stabilization of its stable module category; see Theorem~\ref{thm:BKVB}. The Gorenstein case of the result is due to \cite[Theorem~6.5.3]{Buc}, and the general case is due to \cite[Section~2]{KV}. We mention the treatment in \cite[Subsection~3.1]{Bel}, which seems leaving  out some subtle details.

 Let $R$ be an arbitrary unital ring. Denote by $R\mbox{-Mod}$ the category of all left $R$-modules, and by $R\mbox{-Proj}$ the full subcategory formed by projective $R$-modules. The factor category of $R\mbox{-Mod}$ modulo all morphisms factoring through projective modules is denoted by  $R\mbox{-\underline{Mod}}$, called the  \emph{stable module category} over $R$. The Hom-group between two $R$-modules $M$ and $N$ in $R\mbox{-\underline{Mod}}$ is denoted by ${\rm \underline{Hom}}_R(M, N)$. For any homomorphism $f\colon M \rightarrow N$, the corresponding element in ${\rm \underline{Hom}}_R(M, N)$ is denoted by $\underline{f}$.

 For each $R$-module $M$, we fix a short exact sequence 
 \begin{align}\label{equ:Omega}
 0\longrightarrow \Omega_R(M) \stackrel{\iota_M} \longrightarrow P(M) \stackrel{\pi_M}\longrightarrow M \longrightarrow 0
 \end{align}
 with $P(M)$ projective. This gives rise to the \emph{syzygy endofunctor} $\Omega_R\colon  R\mbox{-\underline{Mod}}\rightarrow R\mbox{-\underline{Mod}}$. Each short exact sequence $\xi \colon 0\rightarrow N \rightarrow E \rightarrow M \rightarrow 0$ in $R\mbox{-Mod}$ fits into a commutative diagram.
 \begin{align}\label{diag:ct}
 \xymatrix{
 \Omega_R(M) \ar[d]_-{h} \ar[rr]^-{\iota_M} && P(M)\ar@{.>}[d]^-{u} \ar[rr]^-{\pi_M} && M \ar@{=}[d] \\
 N \ar[rr]^-g && E \ar[rr]^-f && M 
  }
  \end{align}
  The homomorphism $h$ is not unique in general. However, the corresponding morpshim $\underline{h}$ is unique. The following sequence in $R\mbox{-\underline{Mod}}$
  \begin{align}\label{tri:can}
      \Omega_R(M) \stackrel{\underline{h}} \longrightarrow N \stackrel{\underline{g}} \longrightarrow E \stackrel{\underline{f}} \longrightarrow M
  \end{align}
  is called the \emph{canonical left-triangle} associated to $\xi$. Denote by $\mathcal{E}_R$ the isomorphism closure of such canonical left-triangles.

The following result is due to \cite[Theorem~7.1]{BM}. A similar consideration is traced back to \cite[Section~4]{Hel60}, \cite[Theorem~9.2]{Hel68} and \cite[Theorem~I.2.8]{Hap}.

  \begin{prop}
    The triple $(R\mbox{-\underline{\rm Mod}}, \Omega_R, \mathcal{E}_R)$ is a left triangulated category. \hfill
  \end{prop}

An $R$-module $M$ is said to have a \emph{component-wise  finite resolution} if there is a projective resolution 
$\cdots \rightarrow P^{-2} \rightarrow P^{-1} \rightarrow P^0 \rightarrow M\rightarrow 0$ with each $P^{-i}$ being finitely generated. Denote by $R\mbox{-mod}^{\rm fr}$ the full subcategory consisting of such modules. It is closed under extensions and kernels of surjective homomorphisms. Consequently, the restricted syzygy endofunctor $\Omega_R\colon R\mbox{-\underline{mod}}^{\rm fr}\rightarrow  R\mbox{-\underline{mod}}^{\rm fr}$ is well defined. Furthermore, we have a left triangulated category $(R\mbox{-\underline{mod}}^{\rm fr}, \Omega_R, \mathcal{E}^{\rm fr}_R)$, where $\mathcal{E}^{\rm fr}_R$ is the subclass of $\mathcal{E}_R$ formed by modules having component-wise finite resolutions.  We mention if $R$ is left coherent, we have $R\mbox{-mod}^{\rm fr}=R\mbox{-mod}$, the category of finitely presented $R$-modules.  

The following facts will be used later. 

\begin{rem}\label{rem:Z}
(1)    Denote by $\mathbf{K}(R\mbox{-Proj})$ the homotpy category of cochain complexes of projective $R$-modules. For each complex $P$ and an integer $n$, we denote by $Z^{n}(P)={\rm Ker}(d_P^n\colon P^n\rightarrow P^{n+1})$ the $n$-th cocycle. This gives rise to a well-defined functor $Z^n\colon \mathbf{K}(R\mbox{-Proj})\rightarrow R\mbox{-\underline{\rm Mod}}$. 

(2) For each module $M$, we use the sequence (\ref{equ:Omega}) repeatedly and obtain  a projective resolution $P_M$ such that $P^{0}=P(M)$ and $P^{-n}=P(\Omega^{n}_R(M))$. For any homomorphism $f\colon M\rightarrow N$, we lift it to a morphism $\tilde{f}\colon P_M \rightarrow P_N$ between complexes. We observe that $Z^{-n}(\tilde{f})=\Omega_R^{n+1}(\underline{f})$ for any $n\geq 1$. 
\end{rem}

We denote by $\mathbf{D}^b(R\mbox{-Mod})$ the bounded derived category of $R\mbox{-Mod}$, and by $\mathbf{K}^b(R\mbox{-Proj})$ the bounded homotopy category of $R\mbox{-Proj}$. We view   $\mathbf{K}^b(R\mbox{-Proj})$ as a triangulated subcategory of  $\mathbf{D}^b(R\mbox{-Mod})$. Denote by $R\mbox{-proj}$ the category of finitely generated projective $R$-modules, and by  $\mathbf{K}^b(R\mbox{-proj})$ its bounded homotopy category. Denote by $\mathbf{K}^{-,b}(R\mbox{-proj})$ the homotopy category of right-bounded complexes in $R\mbox{-proj}$ with bounded cohomologies. We denote by $\Sigma$ the suspension functor on complexes. 

The following notions are due to \cite{Buc}, which are rediscovered by \cite{Orl} in the geometric setting; compare \cite{Hap91, Bel}.

\begin{defn}
    The \emph{big singularity category} of $R$ is defined to be the Verdier quotient category $\mathbf{D}'_{\rm sg}(R)=\mathbf{D}^b(R\mbox{-Mod})/\mathbf{K}^b(R\mbox{-Proj})$, and the \emph{singularity category} of $R$ is defined to be $\mathbf{D}_{\rm sg}(R)=\mathbf{K}^{-,b}(R\mbox{-proj})/{\mathbf{K}^{b}(R\mbox{-proj})}$.
\end{defn}

Recall  by definition that the objects in $\mathbf{D}'_{\rm sg}(R)$ are bounded complexes of $R$-modules. We view $R$-modules as stalk complexes concentrated in degree zero. For any subcategory $\mathcal{X}$ in a triangulated category, ${\rm tri}\langle \mathcal{X} \rangle$ denotes the smallest triangulated subcategory containing $\mathcal{X}$.

\begin{lem}\label{lem:fr}
    The singularity category $\mathbf{D}_{\rm sg}(R)$ is triangle equivalent to the smallest triangulated subcategory of $\mathbf{D}'_{\rm sg}(R)$ containing $R\mbox{-{\rm mod}}^{\rm fr}$. 
\end{lem}

      \begin{proof}
          Denote by $\mathbf{K}^{-, b}(R\mbox{-Proj})$ the homotopy category of bounded-above complexes in $R\mbox{-Proj}$ with bounded cohomologies. It is well known that $\mathbf{K}^{-, b}(R\mbox{-Proj})$ is triangle equivalent to $\mathbf{D}^b(R\mbox{-Mod})$; moreover, under the equivalence, $\mathbf{K}^{-, b}(R\mbox{-proj})$ is identified with the subcategory ${\rm tri}\langle  R\mbox{-{\rm mod}}^{\rm fr} \cup \mathbf{K}^{b}(R\mbox{-proj})\rangle$ of $\mathbf{D}^b(R\mbox{-Mod})$. 
          
          We identify $\mathbf{D}'_{\rm sg}(R)$ with $\mathbf{K}^{-,b}(R\mbox{-Proj})/{\mathbf{K}^{b}(R\mbox{-Proj})}$. Therefore, the inclusion $\mathbf{K}^{-,b}(R\mbox{-proj})\subseteq \mathbf{K}^{-,b}(R\mbox{-Proj})$ induces a triangle functor $\mathbf{D}_{\rm sg}(R) \rightarrow \mathbf{D}'_{\rm sg}(R)$. By \cite[Proposition~1.13]{Orl}, this induced functor is fully faithful. Moreover, by the discussion above, we infer that the essential image coincides with ${\rm tri}\langle R\mbox{-{\rm mod}}^{\rm fr} \rangle$. 
      \end{proof}

By identifying modules with the corresponding stalk complexes concentrated in degree zero,  $R\mbox{-Mod}$ is viewed as a full subcategory of $\mathbf{D}^b(R\mbox{-Mod})$. Consider the following composite functor.
$$F\colon R\mbox{-Mod} \hookrightarrow \mathbf{D}^b(R\mbox{-Mod})\xrightarrow{\rm can} \mathbf{D}'_{\rm sg}(R)$$
Here, `${\rm can}$' denotes the quotient functor. The functor $F$ vanishes on projective modules, and induces the following one.
$$F\colon R\mbox{-}\underline{\rm Mod}\longrightarrow \mathbf{D}'_{\rm sg}(R).$$

For any $R$-module $M$, consider the short exact sequence (\ref{equ:Omega}). Denote by ${\rm Cone}(\iota_M)$ the mapping cone of $\iota_M$, which is supported on degrees $-1$ and $0$. The epimorphism $\pi_M$ induces a quasi-isomorphism
$$(0, \pi_M)\colon {\rm Cone}(\iota_M) \longrightarrow M.$$
The projection ${\rm pr}_M\colon {\rm Cone}(\iota_M)\rightarrow \Sigma\Omega_R(M)$ becomes an isomorphism in $\mathbf{D}'_{\rm sg}(R)$, since its cone is isomorphic to $\Sigma(P(M))$ and belongs to $\mathbf{K}^b(R\mbox{-Proj})$. Therefore, the following isomorphism
$$\delta_M=-\Sigma^{-1}(({\rm pr}_M)^{-1}\circ (0, \pi_M))\colon F\Omega_R(M)=\Omega_R(M) \longrightarrow \Sigma^{-1}(M)=\Sigma^{-1}F(M)$$
in $\mathbf{D}'_{\rm sg}(R)$ is well defined. In summary, we have a looped functor 
\begin{align}\label{equ:F-delta}
    (F, \delta)\colon (R\mbox{-}\underline{\rm Mod}, \Omega_R)\longrightarrow (\mathbf{D}'_{\rm sg}(R), \Sigma^{-1}).
\end{align}

\begin{lem}
    The looped functor $(F, \delta)$ above is a triangle functor from the left triangulated category $R\mbox{-}\underline{\rm Mod}$ to $\mathbf{D}'_{\rm sg}(R)$. 
\end{lem}

\begin{proof}
    It suffices to show that  $(F, \delta)$ sends canonical left-triangles in $R\mbox{-}\underline{\rm Mod}$ to exact triangles in $\mathbf{D}'_{\rm sg}(R)$. Take any short exact sequence $\xi \colon 0\rightarrow N \stackrel{g}\rightarrow E \stackrel{f} \rightarrow M \rightarrow 0$   of $R$-modules. It is well known that $\xi$ fits into the following exact triangle in $\mathbf{D}^b(R\mbox{-Mod})$ and thus in $\mathbf{D}'_{\rm sg}(R)$.
    \begin{align}\label{tri:insg}
        N \stackrel{g} \longrightarrow E \stackrel{f} \longrightarrow M \xrightarrow{{\rm pr}\circ {(0, f)}^{-1}} \Sigma(N) 
    \end{align}
    Here, ${\rm pr}\colon {\rm Cone}(g) \rightarrow \Sigma(N)$ is the projection, and $(0, f)\colon {\rm Cone}(g)\rightarrow M$ is the quasi-isomorphism induced by $f$. The looped functor $(F, \delta)$ sends (\ref{tri:can}) to the following sequence.
    \begin{align}\label{tri:F}
          \Sigma^{-1}F(M) \xrightarrow{F(\underline{h})\circ \delta_M^{-1}} F(N) \stackrel{F(\underline{g})} \longrightarrow F(E) \stackrel{F(\underline{f})} \longrightarrow F(M)
    \end{align}
    Recall that $F(\underline{g})=g$ and $F(\underline{f})=f$ in $\mathbf{D}'_{\rm sg}(R)$. By comparing (\ref{tri:insg}) and (\ref{tri:F}), it remains to prove the following identity, again in $\mathbf{D}'_{\rm sg}(R)$.
    $$F(\underline{h})\circ \delta_M^{-1}=-\Sigma^{-1}({\rm pr}\circ {(0, f)}^{-1}).$$
    However, the identity above follows immediately from the fact $F(\underline{h})=h$ and the following commutative diagram  of complexes of modules. 
\[\xymatrix{
M\ar@{=}[d]  && \ar[ll]_-{(0, \pi_M)} {\rm Cone}(\iota_M) \ar[d]^-{(h, u)} \ar[rr]^-{{\rm pr}_M} && \Sigma\Omega_R(M)\ar[d]^-{\Sigma(h)}\\
M  && \ar[ll]_-{(0, f)} {\rm Cone}(g) \ar[rr]^-{{\rm pr}} && \Sigma(N)
}\]
Here, $(h, u)\colon {\rm Cone}(\iota_M)\rightarrow {\rm Cone}(g)$ is the quasi-isomorphism induced by $h$ and $u$ in (\ref{diag:ct}).  Then  we are done. 
\end{proof}

In what follows, we consider the following stabilizations.
$$\mathcal{S}=(\mathcal{S}(R\mbox{-}\underline{\rm Mod}), \Omega_R), \Sigma, \tilde{\mathcal{E}}_R)  
\mbox{ and } 
\mathcal{S}^{\rm fr}=(\mathcal{S}(R\mbox{-}\underline{\rm mod}^{\rm fr}, \Omega_R), \Sigma, \tilde{\mathcal{E}}^{\rm fr}_R) $$
By Proposition~\ref{prop:stab-tri}, both are triangulated categories.   The following fundamental result is due to \cite{Buc, KV, Bel}.

\begin{thm}\label{thm:BKVB}
    Let $R$ be an arbitrary ring. Then the triangle functor $(F, \delta)$ above induces a triangle equivalence
    $$\tilde{F}\colon \mathcal{S}\stackrel{\sim}\longrightarrow \mathbf{D}'_{\rm sg}(R),$$
which restricts to a triangle equivalence  $\mathcal{S}^{\rm fr}\stackrel{\sim}\rightarrow\mathbf{D}_{\rm sg}(R)$. 
\end{thm}

\begin{proof}
By Corollary~\ref{cor:trifun}, we have the strictly triangle functor $\tilde{F}$ above.  In what follows, we will apply Proposition~\ref{prop:equiv} to show that $\tilde{F}$ is dense, faithful and full. 

For any bounded complex $Z$, we take a projective resolution $P$ which belongs to $\mathbf{K}^{-, b}(R\mbox{-Proj})$. For $n$ large enough, its good truncation $$\tau_{\geq -n}(P)= \cdots 0\longrightarrow Z^{-n}(P) \hookrightarrow P^{-n}\stackrel{d_P^{-n}}\longrightarrow P^{-n+1} \longrightarrow \cdots$$
is quasi-isomorphic to $Z$. The projection $\tau_{\geq -n}(P)\rightarrow \Sigma^{n+1}(Z^{-n}(P))$ is an isomorphism in $\mathbf{D}'_{\rm sg}(R)$, since its mapping cone belongs to $\mathbf{K}^b(R\mbox{-Proj})$. It follows that $Z$ is isomorphic to $\Sigma^{n+1}(Z^{-n}(P))$. By Proposition~\ref{prop:equiv}(3), we infer that $\tilde{F}$ is dense. 

Assume that $f\colon M\rightarrow N$ is a homomorphism of $R$-modules satisfying $F(\underline{f})=0$ in $\mathbf{D}'_{\rm sg}(R)$. Therefore, $f$ factors through a bounded complex $P$ of projective modules. Consider the lift $\tilde{f}\colon P_M \rightarrow P_N$ of $f$ in Remark~\ref{rem:Z}(2). It follows that $\tilde{f}$ factors through $P$ in $\mathbf{K}(R\mbox{-Proj})$. For any integer $n$, we apply  the functor $Z^{-n}$ in Remark~\ref{rem:Z}(1) to $\tilde{f}$. The factorization implies that $Z^{-n}(\tilde{f})=0$ for $n$ large enough. By Remark~\ref{rem:Z}(2) again, we infer that $\Omega_R^n(\underline{f})=0$.   By Proposition~\ref{prop:equiv}(2), we infer that $\tilde{F}$ is faithful. 

The proof of fullness is slightly  more involved. Take two modules $M$ and $N$, and a morphism $g\colon M\rightarrow N$ in $\mathbf{D}'_{\rm sg}(R)$. By the very definition of the Verdier quotient, $g=b\circ s^{-1}$ is given by the following roof.
$$M \stackrel{s}\longleftarrow Y \stackrel{b}\longrightarrow N$$
Here, $Y$ is a bounded complex, both $s$ and $b$ are morphisms in $\mathbf{D}^b(R\mbox{-Mod})$ such that ${\rm Cone}(s)$ belongs to $\mathbf{K}^b(R\mbox{-Proj})$. Take a projective resolution $\pi\colon P\rightarrow Y$. The surjective homomorphism $\pi_M$ induces a quasi-isomorphism $\pi_M\colon P_M\rightarrow M$. There is a cochain map $\tilde{s}\colon P\rightarrow P_M$ such that $\pi_M\circ \tilde{s}=s\circ \pi$. Similarly, there is a cochain map $\tilde{b}\colon P\rightarrow P_N$ satisfying $\pi_N\circ \tilde{b}=b\circ \pi$. 

The cone of $\tilde{s}$ is isomorphic to a bounded complex of projective modules. It follows that 
$$Z^{-n}(\tilde{s})\colon  Z^{-n}(P)  \longrightarrow Z^{-n}(P_M)=\Omega^{n+1}_R(M)$$
is an isomorphism for $n$ sufficiently large. We fix a sufficiently large $n$. Consider the following commutative diagram, whose upper row also represents the morphism $g$ as a roof.
\[\xymatrix@C=18pt{
M & \ar[l]_-{\pi_M} \tau_{\geq -n}(P_M) \ar[d]_-{{\rm pr}} &&\ar[ll]_-{\tau_{\geq -n}(\tilde{s})} \tau_{\geq -n}(P) \ar[d]^-{\rm pr} \ar[rr]^-{\tau_{\geq -n}(\tilde{b})} && \tau_{\geq -n}(P_N) \ar[r]^-{\pi_N}\ar[d]^-{\rm pr} & N\\
& \Sigma^{n+1}\Omega_R^{n+1}(M) && \ar[ll]_-{\Sigma^{n+1}Z^{-n}(\tilde{s})} \Sigma^{n+1}Z^{-n}(P) \ar[rr]^-{\Sigma^{n+1}Z^{-n}(\tilde{b})} && \Sigma^{n+1} \Omega_R^{n+1}(N)
}\]
Here, the three ``${\rm pr}$" denote the relevant projections. Set $$\underline{f}=Z^{-n}(\tilde{b})\circ Z^{-n}(\tilde{s})^{-1}\colon \Omega_R^{n+1}(M) \longrightarrow  \Omega_R^{n+1}(N).$$ 
In this diagram, we observe that 
$${\rm pr}\circ \pi_M^{-1}=(-1)^{n+1}\; \Sigma^{n+1}(\delta^{(n+1)}_M)^{-1} \mbox{ and } \pi_N\circ {\rm pr}^{-1} =(-1)^{n+1}\; \Sigma^{n+1}(\delta^{(n+1)}_N).$$
Combining these facts, we infer that 
$$g=\Sigma^{n+1}(\delta^{(n+1)}_N) \circ \Sigma^{n+1}(F(f))\circ \Sigma^{n+1}(\delta^{(n+1)}_M)^{-1}. $$
Now, we apply Proposition~\ref{prop:equiv}(1) to infer that $\tilde{F}$ is full. 

For the restricted equivalence, we observe that $\mathcal{S}^{\rm fr}$ is a triangulated subcategory of $\mathcal{S}$. Then the result follows immediately from Lemma~\ref{lem:fr}. 
\end{proof}

\section{Singular equivalences}

In this section, we collect several known results on singular equivalences, which are obtained using the stabilization. 

Let $R$ be a left coherent ring. Then the category $R\mbox{-mod}$ of finitely presented modules is abelian. By the canonical equivalence between $\mathbf{K}^{-, b}(R\mbox{-proj})$ and $\mathbf{D}^b(R\mbox{-mod})$, the singularity category $\mathbf{D}_{\rm sg}(R)$ might be defined by $\mathbf{D}^b(R\mbox{-mod})/{\mathbf{K}^b(R\mbox{-proj})}$.

Recall that a full additive subcategory $\mathcal{X}\subseteq R\mbox{-mod}$ is called \emph{quasi-resolving} \cite{BST} if $\mathcal{X}$ contains $R\mbox{-proj}$ and is closed under kernels of epimorhisms. It follows that the corresponding full subcategory $\underline{\mathcal{X}}\subseteq R\mbox{-}\underline{\rm mod}$ is also left triangulated. A quasi-resolving subcategory $\mathcal{X}$ is \emph{ample} if for each $R$-module $M$, there exists a natural number $n$ such that $\Omega^n_R(M)$ belongs to $\mathcal{X}$. The following immediate consequence of Theorem~\ref{thm:BKVB} is due to \cite[Theorem~3.2]{BST}.

\begin{prop}\label{prop:quasi-res}
  Let $\mathcal{X}$ be a quasi-resolving subcategory. Then the functor $F$ in (\ref{equ:F-delta}) induces a fully faithful triangle functor 
  $$\tilde{F}\colon \mathcal{S}(\underline{\mathcal{X}})\longrightarrow \mathbf{D}_{\rm sg}(R).$$
  Moreover, this induced functor $\tilde{F}$ is dense if and only if $\mathcal{X}$ is ample.  
\end{prop}

\begin{proof}
    We observe that $\mathcal{S}(\underline{\mathcal{X}})$ is a triangulated subcategory of $\mathcal{S}(R\mbox{-}\underline{\rm mod})$. Then the first statement follows from Theorem~\ref{thm:BKVB}. The last statement follows from \cite[Corollary~2.3]{Chen18}. In this situation,  the inclusion  $\underline{\mathcal{X}}\subseteq R\mbox{-}\underline{\rm mod}$ is a pre-triangle equivalence in the sense of Definition~\ref{defn:pte}.
\end{proof}

Recall that a unbounded complex $P$ of finitely generated projective $R$-modules is \emph{totally acyclic} if it is acyclic and its dual complex ${\rm Hom}_R(P,R)$ is also acyclic. An $R$-module $G$ is \emph{Gorenstein-projective} \cite{EJ} if there is a totally acyclic complex $P$ with $Z^{1}(P)\simeq G$; compare \cite{ABr}. Denote by $R\mbox{-Gproj}$ the full subcategory of Gorenstein-projective modules; it is a quasi-resolving subcategory.  Since Gorenstein-projective modules are closed under extensions, it becomes an exact category in the sense of Quillen. Moreover, it is Frobenius whose projective-injective objects are precisely finitely generated projective $R$-modules. It follows from \cite[Theorem~I.2.6]{Hap} that the stable category $R\mbox{-}\underline{\rm Gproj}$ is naturally triangulated.  

Recall that the ring $R$ is \emph{left $G$-regular} provided that the subcategory $R\mbox{-Gproj}$  is ample. Examples are Gorenstein rings, that is, two-sided noetherian rings with finite selfinjective dimension on both sides.

The following well-known result is due to \cite[Theorem~4.4]{Buc}; compare \cite[Theorem~4.6]{Hap91} and \cite[Theorem~6.9]{Bel}. We refer to \cite[Section~4]{CW} for another proof. The equivalence (\ref{equ:Ric}) below is also due to \cite[Theorem~2.1]{Ric89}.

\begin{thm}{\rm (Buchweitz)}\label{thm:Buc}
Let $R$ be a left coherent ring. Then    the functor $F$ in (\ref{equ:F-delta}) restricts to  a fully faithful triangle functor 
  $$F\colon R\mbox{-}\underline{\rm Gproj} \longrightarrow \mathbf{D}_{\rm sg}(R).$$
  Moreover, this restricted functor is dense if and only if $R$ is left $G$-regular. Consequently, if $R$ is quasi-Frobenius, the restricted functor yields a triangle equivalence
  \begin{align}\label{equ:Ric}
      F\colon R\mbox{-}\underline{\rm mod} \stackrel{\sim}\longrightarrow \mathbf{D}_{\rm sg}(R).
  \end{align}
\end{thm}

\begin{proof}
    Since  $R\mbox{-}\underline{\rm Gproj}$ is already triangulated, we identify it with its stabilization $\mathcal{S}(R\mbox{-}\underline{\rm Gproj})$. Then the results follow from Proposition~\ref{prop:quasi-res}. For the last statement, we just observe that $R\mbox{-Gproj}=R\mbox{-mod}$ for a quasi-Frobenius ring $R$. 
\end{proof}

In general, the restricted functor $F$ above is not dense. Following \cite{BOJ}, the Verdier quotient category $\mathbf{D}_{\rm sg}(R)/{{\rm Im}\; F}=\mathbf{D}_{\rm def}(R)$ is called  the \emph{Gorenstein defect category} of $R$, which detects the Gorensteinness of the ring;  compare \cite[Theorem~4.2]{BOJ}.

The following example plays a central role in this section.

\begin{exm}
Let $A$ be a left artinian ring whose Jacobson radical is square zero, that is, ${\bf r}^2={\rm rad}(A)^2=0$. Denote by $A\mbox{-}{\rm spmod}={\rm add}(A/{\bf r}\oplus A)\subseteq A\mbox{-}{\rm mod}$. It is an ample quasi-resolving subcategory; see \cite[Section~3]{Chen11}. Consequently, we have a triangle equivalence
    \begin{align}\label{tri:radical}
        \mathcal{S}(A\mbox{-}\underline{\rm spmod}) \stackrel{\sim}\longrightarrow \mathbf{D}_{\rm sg}(A).
    \end{align}
    Since $A\mbox{-}\underline{\rm spmod}$  is semisimple, it follows from Example~\ref{exm:ss} that $\mathbf{D}_{\rm sg}(A)$ is also semisimple.
    
By identifying $A/{\bf r}$-modules with semisimple $A$-modules, we have an obvious   functor $A/{\bf r}\mbox{-}{\rm mod}\rightarrow A\mbox{-}\underline{\rm spmod}$, which is full and dense. By \cite[Lemma~3.1 and the proof of Proposition~3.2]{Chen11}, it gives rise to a pre-triangle equivalence between the two semisimple left triangulated categories $(A/{\bf r}\mbox{-}{\rm mod}, \mathbf{r}\otimes_{A/{\bf r}}-)$ and $(A\mbox{-}\underline{\rm spmod}, \Omega_A)$. Consequently, we have a triangle equivalence
\begin{align}\label{tri:radical2}
 \mathcal{S}(A/{\bf r}\mbox{-}{\rm mod}, {\bf r}\otimes_{A/{\bf r}}-) \stackrel{\sim}\longrightarrow \mathbf{D}_{\rm sg}(A).
\end{align}    
\end{exm}

Let us mention examples of rings with radical square zero. Let $\mathbb{K}$ be a field. Assume that $D$ is a finite dimensional semisimple $\mathbb{K}$-algebra and that $M$ is a finitely generated $D$-$D$-bimodule, on which $\mathbb{K}$ acts centrally. Then  the trivial extension algebra $B=D\ltimes M$ is such an example. 

Let $Q=(Q_0, Q_1; s,t)$ be a finite quiver. Here,  $Q_0$ is the set of vertices, $Q_1$  is the set of arrows and for each arrow $\alpha$, $s(\alpha)$ and $t(\alpha)$ denote its starting vertex and terminating vertex, respectively. Denote by $\mathbb{K}Q$ the path algebra. The trivial extension algebra $\mathbb{K}Q_0\ltimes \mathbb{K}Q_1$ is identified with the quotient algebra $\mathbb{K}Q/{J^2}$, where $J$ is the two-sided ideal generated by arrows.

\begin{exm}
    Let $R$ be a commutative local Cohen-Macaulay ring, and $\Lambda$ an \emph{$R$-order}, that is, $\Lambda$ is an $R$-algebra whose underlying $R$-module is finitely generated and maximal Cohen-Macaulay, for example, free of finite rank. An $\Lambda$-module is defined to be \emph{CM} if the underlying $R$-module is maximal Cohen-Macaulay. Denote by ${\rm CM}(\Lambda)$  the full subcategory formed by CM $\Lambda$-modules. It is an ample quasi-resolving subcategory of $\Lambda\mbox{-}{\rm mod}$; compare \cite[Proposition~4.3]{BST}. Consequently, we have a triangle equivalence
\begin{align}\label{tri:CM}
        \mathcal{S}(\underline{\rm CM}(\Lambda)) \stackrel{\sim}\longrightarrow \mathbf{D}_{\rm sg}(\Lambda).
    \end{align}    
\end{exm}

    Recall that a \emph{derived equivalence} \cite{Ric} between $R$ and $S$ means a triangle equivalence $\mathbf{D}(R\mbox{-Mod})\rightarrow \mathbf{D}(S\mbox{-Mod})$ between the unbounded derived category. It is an important subject in representation theory \cite{Rou}.  Analogously, a \emph{singular equivalence} between two left coherent rings $R$ and $S$ is defined to be a triangle equivalence between  $\mathbf{D}_{\rm sg}(R)$ and $\mathbf{D}_{\rm sg}(S)$.     In view of \cite[Lemma~5.4]{Kr}, any derived equivalence between $R$ and $S$ induces a singular equivalence. 

    \begin{rem}
    (1) Assume that $F\colon \mathbf{D}^b(R\mbox{-mod})\rightarrow \mathbf{D}^b(S\mbox{-mod})$ is a triangle equivalence for two left coherent rings $R$ and $S$. By \cite[Proposition~3.6]{JKS},  $F$ necessarily sends $\mathbf{K}^b(R\mbox{-proj})$ to $\mathbf{K}^b(S\mbox{-proj})$; compare \cite[Proposition~8.2]{Ric} and \cite[Theorem~5.1]{CNS}. It follows that $F$ induces a singular equivalence between $R$ and $S$.
    
    (2) In view of (\ref{equ:Ric}), we conclude that singular equivalences between rings are natural extensions of \emph{stable equivalences} \cite{Ric89, Broue} between quasi-Frobenius rings.
    \end{rem}

Let $B=\mathbb{K}Q/I$ be a quadratic monomial algebra, for example, a gentle algebra. Here, $Q$ is a finite quiver, and $I\subseteqq \mathbb{K}Q$ is an admissible ideal which is generated by some paths of length two. The \emph{relation quiver} $Q^r$ of $B$ is defined as follows. The set $Q^r_0$ of vertices is defined  to be $Q_1$. For each relation $\beta\alpha$ in $I$, we set an arrow $[\beta\alpha]$ from $\alpha$ to $\beta$.   In other words, $Q^r_1$ is given by the set of  generators in $I$. The following result is due to \cite[Theorem~4.5]{Chen18}; see also \cite[Theorem~6.2]{CLW23}.

\begin{thm}
    Let $B=\mathbb{K}Q/I$ be a quadratic monomial algebra. Then there is a singular equivalence between $B$ and $\mathbb{K}Q^r/{J^2}$. 
\end{thm}

\begin{proof}
Consider the natural $\mathbb{K}$-linear functor 
$$H\colon \mathbb{K}Q^r/{J^2}\mbox{-\underline{spmod}}\longrightarrow B\mbox{-}{\underline{\rm mod}},$$ 
which sends the simple module $S_\alpha$ corresponding to the vertex $\alpha$ in $Q^r$ to the left ideal $B\alpha$ of $B$.  The key ingredient is to show that $H$ is a part of  a pre-triangle equivalence $(H, \delta)$; see \cite[Lemma~4.6]{Chen18}. Then using (\ref{tri:radical}) and Proposition~\ref{prop:quasi-res}, we obtain the required singular equivalence. 
\end{proof}

Consider the power series algebra $\mathbb{C}[[x_1, x_2, x_3]]$ over the complex numbers. It has an algebra automorphism $\sigma$ which sends  $x_i$ to $-x_i$. Denote by $\mathbb{C}[[x_1, x_2, x_3]]^\sigma$ the invariant subalgebra. The following result is due to \cite[Section~3]{Kalck}, which is based on \cite[Part II, Theorem~4.1]{AR}.

\begin{thm}
    There is a singular equivalence between $\mathbb{C}[[x_1, x_2, x_3]]^\sigma$ and $\mathbb{C}K_3/{J^2}$, where $K_3$ is the quiver consisting of a single vertex and three loops. 
\end{thm}

\begin{proof}
The algebra $R=\mathbb{C}[[x_1, x_2, x_3]]^\sigma$ is Cohen-Macaulay. There is a $\mathbb{C}$-linear functor 
$$H\colon \mathbb{C}K_3/{J^2}\mbox{-}\underline{\rm spmod}\longrightarrow  \underline{\rm CM}(R),$$
which sends the unique simple module $\mathbb{C}$ over $\mathbb{C}K_3/{J^2}$  to $\Omega_R(\omega)$, the syzygy of the canonical module $\omega$ of $R$. By \cite[Propositions~3.1 and 3.2]{Kalck}, the functor $H$ is a part of a pre-triangle equivalence $(H, \delta)$. In view of (\ref{tri:radical}) and (\ref{tri:CM}), we infer the required singular equivalence.
\end{proof}

Let $\mathcal{O}$ be a complete discrete valuation ring with $\mathcal{O}/{{\rm rad}(\mathcal{O})}=\mathbb{K}$. Denote by $K$ the fraction field of $\mathcal{O}$. An $\mathcal{O}$-order $\Lambda$ is called a \emph{B\"{a}ckstr\"{o}m order}  \cite{RR} provided that $K\otimes_\mathcal{O} \Lambda$ is semisimple and that there is an overorder $\Gamma$ of $\Lambda$ in $ K\otimes_\mathcal{O} \Lambda$, which is hereditary with ${\rm rad}(\Lambda)={\rm rad}(\Gamma)$.  The following result is due to \cite[Theorem~4.4 and Corollary~4.5]{Wei}.

\begin{thm}
    Let $\Lambda$ be a B\"{a}ckstr\"{o}m order. Then there is a singular equivalence between $\Lambda$ and $D\ltimes M$ for some finite dimensional semisimple $\mathbb{K}$-algebra $D$ and a $D$-$D$-bimodule $M$. 
\end{thm}

\begin{proof}
    Set $D=\underline{\rm End}_\Lambda(\Gamma)$ and $M=\underline{\rm Hom}_\Lambda(\Gamma, \Omega_\Lambda(\Gamma))$. By \cite[Proposition~4.2]{Wei}, $D$ is a finite dimensional semisimple $\mathbb{K}$-algebra. By \cite[Proposition~4.3]{Wei}, $M$ is naturally a $D$-$D$-bimodule.  We have an obvious  additive functor 
    $$H\colon D \mbox{-}{\rm mod}\longrightarrow  \underline{\rm CM}(\Lambda),$$
    sending $D$ to $\Gamma$. Here, we view $(D \mbox{-}{\rm mod}, M\otimes_D-)$ as a semisimple left triangulated category; see Example~\ref{exm:ss}.  By \cite[Proposition~4.1 and (4.3)]{Wei}, the functor $H$ is a part of a pre-triangle equivalence $(H, \delta)$.  By combining this pre-triangle equivalence with (\ref{tri:radical2}) and (\ref{tri:CM}), we obtain the required singular equivalence. 
\end{proof}

\section{Leavitt rings}

In this final section, we report a recent characterization \cite{CW25} of the singularity category of a certain artinian ring in terms of graded modules over Leavitt rings.

Let $R$ be an arbitrary unital ring, and $_RM_R$ an $R$-$R$-bimodule such that its underlying  left $R$-module $_RM$ is finitely generated projective. Then we have a looped category $(R\mbox{-mod}, M\otimes_R-)$, and form the stabilization $\mathcal{S}(R\mbox{-mod}, M\otimes_R-)$.

Denote by $M^*={\rm Hom}_R({_R}M, {_R}R)$ the \emph{left-dual bimodule}. Its $R$-$R$-bimodule structure is given such that
$$(af)(x)=f(xa) \mbox{ and } (fb)(x)=f(x)b$$
for any $a, b\in R$, $x\in M$ and $f\in M^*$. Recall that the \emph{Casimir element} $c\in M^*\otimes_R M$ is the image of ${\rm Id}_M$ under the canonical isomorphism ${\rm End}_R(M)\stackrel{\sim}\rightarrow M^*\otimes_R M$.

Recall from \cite[Definition~2.4]{CWKW} that the \emph{Leavitt ring} associated to the $R$-$R$-bimodule $M$ is defined to be
$$L_R(M)=T_R(M^*\oplus M)/{(x\otimes_R f-f(x), c-1_R \; |\; x\in M, f\in M^*)}.$$
Here, $T_R(M^* \oplus M)$ denotes the tensor ring of the bimodule $M^*\oplus M$. The Leavitt ring is $\mathbb{Z}$-graded such that ${\rm deg}(x)=-1$ and ${\rm deg}(f)=1$ for $x\in M$ and $f\in M^*$. We mention that Leavitt rings are certain Cuntz-Primsner rings in the sense of \cite[Definition~3.16]{CO}.

\begin{exm}
    Let $\mathbb{K}$ be a field and $Q$ be a finite quiver. Denote by $Q^\circ$ the quiver without sinks, which is obtained from $Q$ by repeatedly removing all possible sinks. Then  by \cite[Proposition~4.1(2)]{CWKW} the Leavitt ring $L_{\mathbb{K}Q_0}(\mathbb{K}Q_1)$ is isomorphic to the \emph{Leavitt path algebra} $L(Q^\circ)$ in the sense of \cite{AAP, AMP}. We mention that Leavitt path algebras are natural generalizations of Leavitt algebras in \cite{Leav}. 
\end{exm}

For any $\mathbb{Z}$-graded ring $\Gamma$, we denote by $\Gamma\mbox{-grmod}$ the category of finitely presented graded $\Gamma$-modules, and by $\Gamma\mbox{-grproj}$ the full subcategory formed by finitely generated graded projective $\Gamma$-modules.  We have the stable category  $\Gamma\mbox{-\underline{grmod}}$.  The following result  is due to \cite[Theorem~4.4]{CW25}, which connects Leavitt rings with   stabilizations of module categories.

\begin{thm}\label{thm:s-gr}
Let $R$ be a ring and $M$ be an $R$-$R$-bimodule such that $_RM$ is finitely generated projective. Then we have an equivalence
$$\mathcal{S}(R\mbox{-{\rm mod}}, M\otimes_R-)\simeq L_R(M)\mbox{-{\rm grmod}},$$
which induces an equivalence
$$\mathcal{S}(R\mbox{-{\rm \underline{mod}}}, M\otimes_R-)\simeq L_R(M)\mbox{-}{\rm \underline{grmod}}.$$
\end{thm}

\begin{rem}\label{rem:L}
Set $\mathcal{S}=\mathcal{S}(R\mbox{-{\rm mod}}, M\otimes_R-)$. Denote by $(1)$ the degree-shift automorphism on $L_R(M)\mbox{-{\rm grmod}}$.  The equivalence above is indeed a strictly looped functor between $(\mathcal{S}, \Sigma)$ and $(L_R(M)\mbox{-{\rm grmod}}, (1))$; see the proof of \cite[Proposition~2.2]{CW25}.
\end{rem}

Let $A$ be a left artinian ring with radical square zero. Denote by ${\bf r}={\rm rad}(A)$. Then ${\bf r}$ is naturally an $A/{\bf r}$-$A/{\bf r}$-bimodule.  By \cite[Lemma~3.8 and Proposition~4.3]{CW25} the Leavitt ring $L_{A/{\bf r}}({\bf r})$ is \emph{graded von Neumann regular}, that is, it is strongly graded whose zeroth component $L_{A/{\bf r}}({\bf r})_0$ is von Neumann regular. It follows that $L_{A/{\bf r}}({\bf r})\mbox{-}{\rm grproj}=L_{A/{\bf r}}({\bf r})\mbox{-}{\rm grmod}$, which is a semisimple abelian category. By \cite[Lemma~3.4]{Chen11},  we will view $L_{A/{\bf r}}({\bf r})\mbox{-}{\rm grproj}$ as a semisimple triangulated category with the suspension functor given by $(1)$. The following result slightly generalizes \cite[Theorem~5.9]{Smi}; compare \cite[Theorem~B]{Chen11} and \cite[Theorem~6.1]{CY}. 

\begin{thm}\label{thm:radical}
Let $A$ be a left artinian ring with radical square zero. Then there is a triangle equivalence
$$\mathbf{D}_{\rm sg}(A)\simeq L_{A/{\bf r}}({\bf r})\mbox{-}{\rm grproj}.$$
\end{thm}

\begin{proof}
 Recall that $L_{A/{\bf r}}({\bf r})\mbox{-}{\rm grproj}=L_{A/{\bf r}}({\bf r})\mbox{-}{\rm grmod}$. Now the required equivalence follows from (\ref{tri:radical2}) and Theorem~\ref{thm:s-gr}. Here, we use Remark~\ref{rem:L} and implicitly use the uniqueness of the triangulated structure on any semisimple category with a given suspension functor.
\end{proof}

In what follows,  we fix a ring $\Lambda$, which contains a semisimple subring $E$ such that the left $E$-module $_E\Lambda$ is finitely generated. It follows that $\Lambda$ is left artinian. For example, if $\Lambda$ is a finite dimensional $\mathbb{K}$-algebra, we can take $E$ to be  $\mathbb{K}$.

Consider the quotient $E$-$E$-bimodule $\overline{\Lambda}=\Lambda/E$. An element $a\in \Lambda$ corresponds to $\overline{a}\in \overline{\Lambda}$. The bimodule of \emph{$E$-relative noncommutative differential $1$-forms} \cite{CQ} is defined to be the following $\Lambda$-$\Lambda$-bimodule
$$\Omega_{\rm nc}=\Omega_{{\rm nc}, \Lambda/E}=\overline{\Lambda}\otimes_E \Lambda.$$
Its right $\Lambda$-action is given by $(\overline{a}\otimes_E x)b=\overline{a}\otimes_E xb$, while its left $\Lambda$-action is somewhat nontrivial and given by $b(\overline{a}\otimes_E x)=\overline{ba}\otimes_E x-\overline{b}\otimes_E ax$. By \cite[Proposition~2.5]{CQ}, $\Omega_{\rm nc}$ is projective on both sides.

Following \cite[Definition~2.1]{Dam}, a ring $R$ is calld \emph{FC} if it is two-sided coherent and  satisfies
$${\rm Ext}_R^1(M, R)=0={\rm Ext}^1_{R^{\rm op}}(N, R)$$
for any finitely presented left $R$-module $M$ and finitely presented right $R$-module $N$. This terminology is justified by the fact that flat modules and coflat modules coincide for FC rings. In this case, $R\mbox{-mod}$ is a Frobenius abelian category; see \cite[Lemma~5.1]{CW25}. FC rings are coherent analogues of quasi-Frobenius rings. 

A $\mathbb{Z}$-graded ring $\Gamma$ is called \emph{graded FC}, if it is strongly graded and its zeroth component $\Gamma_0$ is an FC ring. In this situation, the category $\Gamma\mbox{-grmod}$ is Frobenius abelian, and thus its stable category $\Gamma\mbox{-}\underline{\rm grmod}$ is naturally triangulated.  The following result is due to \cite[Theorem~5.7]{CW25}, which realizes the singularity category as a stable module category.

\begin{thm}\label{thm:Leavitt}
    Set $\Omega_{\rm nc}=\Omega_{{\rm nc},\Lambda/E}$. Then the Leavitt ring $L_\Lambda(\Omega_{\rm nc})$ is graded FC. Moreover, we have a triangle equivalence
    $$\mathbf{D}_{\rm sg}(\Lambda)\simeq L_\Lambda(\Omega_{\rm nc})\mbox{-\underline{\rm grmod}}.$$
\end{thm}

The left artinian ring $\Lambda$ has finite global dimension if and only if the  ring $L_\Lambda(\Omega_{\rm nc})$ is graded von Neumann regular. In a certain sense,  the non-regularity of $L_\Lambda(\Omega_{\rm nc})$ detects the homological singularity of $\Lambda$.

\vskip 5pt

\noindent{\bf Acknowledgements}. \quad  We thank Hongxing Chen,  Xiaofa Chen, Henning Krause, Jian Liu and Hongrui Wei for  helpful discussions and suggestions.  This project is supported by National Key R$\&$D Program of China (No.s 2024YFA1013801), and the National Natural Science Foundation of China (No.s 12325101, and 12131015).

\bibliography{}

\vskip 10pt

 {\footnotesize \noindent  Xiao-Wu Chen\\
 School of Mathematical Sciences, University of Science and Technology of China, Hefei 230026, Anhui, PR China}

\end{document}